\documentclass[11pt]{article}
\usepackage[utf8]{inputenc}
\usepackage[T1]{fontenc}
\usepackage{authblk}  
\usepackage{hyperref}  
\usepackage[margin=1.3in]{geometry}  
\usepackage{enumitem}  
\usepackage{graphicx}  
\usepackage{xcolor}  
\usepackage{xspace}  
\usepackage{amssymb, amsfonts, amsmath, amsthm}  
\usepackage{thmtools, thm-restate}  
\usepackage{old-arrows}
\usepackage{marginnote}  

\definecolor{ured}{RGB}{239, 118, 122}  
\definecolor{ugrey}{RGB}{69, 105, 144}  
\definecolor{ugreen}{RGB}{73, 220, 177}  
\definecolor{uyellow}{RGB}{238, 184, 104}  
\definecolor{ublue}{RGB}{0, 180, 216}  

\hypersetup{
	colorlinks=true,
	citecolor=ugreen,
	urlcolor=ugreen,
	linkcolor=ugreen
}

\newcommand{\st}{:}  
\newcommand{\U}{X}  
\newcommand{\card}[1]{|#1|}  
\newcommand{\leftsize}[1]{||#1||_\ell}  
\newcommand{\rightsize}[1]{||#1||_r}  
\newcommand{\totalsize}[1]{||#1||}  
\newcommand{\pow}[1]{\mathcal{P}(#1)}  
\newcommand{\cl}{\phi}  
\newcommand{\cs}{\mathcal{C}}  
\newcommand{\imp}{\rightarrow}  
\newcommand{\ib}{\Sigma}  
\renewcommand{\H}{\mathcal{H}}  
\renewcommand{\emptyset}{\varnothing}  
\DeclareMathOperator{\ex}{ex}  

\newcommand{\HQC}{\mathcal{Q}}  
\newcommand{\lset}[1]{\mathit{#1}} 

\newcommand{\NP}{\mathsf{NP}}  

\theoremstyle{plain}  
\declaretheorem[name=Theorem]{theorem}
\declaretheorem[name=Lemma]{lemma}
\declaretheorem[name=Proposition]{proposition}
\declaretheorem[name=Corollary]{corollary}

\theoremstyle{definition}  

\declaretheorem[name=Remark]{remark}
\declaretheorem[name=Example]{example}

\newif\iflongversion
\longversiontrue

\title{Characterizing the optimum bases of a convex geometry using quasi-closed hypergraphs\footnote{A short version of this paper has been presented at and published in the proceedings of the 3rd International Joint Conference on Conceptual Knowledge Structures (CONCEPTS 2026, see \href{https://link.springer.com/chapter/10.1007/978-3-032-34742-8_5}{link.springer.com/chapter/10.1007/978-3-032-34742-8\_5})}}
\author[1]{Anthony Meunier}
\author[1]{Lhouari Nourine}
\author[2]{Simon Vilmin}

\affil[1]{Université Clermont-Auvergne, CNRS, Clermont-Auvergne-INP, LIMOS, Clermont-Ferrand, France}
\affil[2]{Aix-Marseille Université, CNRS, LIS, Marseille, France.}

\begin{document}
\maketitle

\begin{abstract}
Optimizing an implicational base of a closure system consists in turning this implicational base into an equivalent one with premises and conclusions as small as possible.
This task is known to be hard in general but tractable for a number of classes of closure systems.
In particular, several classes of convex geometries are known to have tractable optimization, while the problem was recently claimed to remain hard in general convex geometries.
Continuing this line of research, we give a characterization of the optimum bases of a convex geometry in terms of what we call quasi-closed hypergraphs.
We then use this characterization to show that when each quasi-closed hypergraph has disjoint edges, any implicational base of the convex geometry can be optimized in polynomial time with existing minimization and reduction algorithms.
Finally, we prove that this property applies to double-shelling convex geometries, $D$-geometries, convex geometries with the $n$-Carousel property and $q$-acceptant convex geometries, thus unifying and extending the existing results regarding the tractability of optimization.
\vspace{0.5em}

\end{abstract}

\section{Introduction} \label{sec:introduction}

Implications, also known as attribute implications in Formal Concept Analysis~\cite{ganter1999formal}, functional dependencies in database theory~\cite{maier1983theory,mannila1992design}, Horn clauses in logic~\cite{boros2009subclass,hammer1995quasi}, or join-covers in lattice theory~\cite{birkhoff1940lattice} to cite a few, are a common representation of finite closure systems.
Indeed, a closure system can always be represented by several (equivalent) implicational bases---i.e., collections of implications.
Within the study of implicational bases, the task of turning an implicational base into a smaller one has been the topic of numerous research.
Several measures of the size of an implicational base have been defined and studied: the number of implications, the sum of sizes of the premises of its implications, of the conclusions, or both.
An implicational base that minimizes the number of implications is \emph{minimum}.
It is \emph{left-optimum} (resp. \emph{right-optimum}) if the sum of sizes of its premises (resp. conclusions) is as small as possible.
It is \emph{optimum} if the sum of the sizes of all its premises and conclusions is minimum among all other equivalent bases.
We refer to Section~\ref{sec:preliminaries} for formal definitions.

In this paper we are interested in the \emph{optimization} problem: given an implicational base of a closure system, find, or turn the implicational base into, an equivalent optimum implicational base.
This problem has been studied from different perspectives such as database theory~\cite{maier1980minimum,ausiello1986minimum,maier1980minimum}, propositional logic~\cite{hammer1993optimal,hammer1995quasi,boros2009subclass,boros2010exclusive,boros2013decomposition,kuvcera2017hydras}, or closure systems and lattices~\cite{wild1994theory,wild2000optimal,nakamura2013prime,adaricheva2014implicational,adaricheva2017optimum,bichoupan2022complexity}.
We review below the main results on optimization and redirect the reader to~\cite[section 6.7]{crama2011boolean} and \cite{wild2017joy,boros2013decomposition} for further references on the topic.
The first result is due to Maier~\cite{maier1980minimum} in his work on covers of functional dependencies.
In that paper, he proves that the decision problem associated to optimization is $\NP$-complete for arbitrary closure systems and that optimum bases must be minimum.
Mannila and Raïha~\cite{mannila1983relationship} then investigate further optimum bases and prove that an optimum implicational base must be left-optimum and that an implicational base is optimum if and only if each of its equivalence class is optimum.
The equivalence classes of an implicational base are the sets of implications whose premises share the same closure. 
This latter result was recently reproved for right-optimization by Bichoupan~\cite[unpublished]{bichoupan2022complexity}, answering an open question of Adaricheva and Nation~\cite{adaricheva2014implicational}, albeit with a different phrasing: the (right-)size of an equivalence class in any optimum implicational base is fixed.
Ausiello et al.~\cite{ausiello1986minimum} (see also \cite{ausiello2017directed}), prove that both left- and right-optimization are $\NP$-complete, which gives a new proof of the hardness of the optimization task.
They also complete the connection between the different minimality measures: an optimum base is both left- and right-optimum.
It thus follows that an implicational base is optimum if and only if it is both left- and right-optimum (see also \cite{adaricheva2014implicational}).
In the language of closure systems and implicational bases, Wild~\cite{wild1994theory,wild2000optimal} shows that both modular closure systems and simple binary matroids can be optimized in polynomial time.
More precisely, for simple binary matroids, he shows that the corresponding closure system does have a unique optimum base that can be found in polynomial time, a result recently rediscovered by Berczi et al.~\cite{berczi2024matroid}.
In~\cite{hammer1993optimal}, Hammer and Kogan give a new proof that optimization is $\NP$-complete, this time in the language of Horn CNFs.
Later~\cite{hammer1995quasi}, they proved that optimization could be achieved in polynomial time for acyclic and quasi-acyclic Horn functions.
In particular, they prove that each acyclic Horn function also has a unique optimum implicational base. 
In \cite{boros2009subclass} Boros et al.\! generalize the results of Hammer and Kogan and show that optimization is tractable for the more general classes of component quadratic Horn functions based on earlier results regarding essential and exclusive sets of implicates~\cite{boros2010exclusive} they used to give new proofs that quadratic, quasi-acyclic and acyclic Horn functions can be optimized.
Then, in another contribution~\cite{boros2013decomposition}, Boros et al.\ proved that optimization remains hard even for implicational bases with premises of size $2$.
This result was later strengthened by Ku{\v{c}}era~\cite{kuvcera2017hydras}, who proves that optimization remains hard even in implicational bases where all premises have size $2$ and the same closure, the ground set.
Besides, in the language of closure systems, Adaricheva and Nation~\cite{adaricheva2014implicational} proved that optimization and right-optimization remain hard even within the class of bounded closure systems.

Convex geometries are yet another classes of closure systems that has recently received attention regarding optimization~\cite{nakamura2013prime,adaricheva2017optimum,bichoupan2022complexity}.
Convex geometries~\cite{edelman1985theory}, closure systems with anti-exchange closure operator, have been studied and rediscovered in many fields ranging from combinatorial optimization to social choice theory (see \cite{monjardet1985use,koshevoy1999choice}).
For optimization, Kashiwabara and Nakamura~\cite{nakamura2013prime} proved that affine convex geometries, that is convex geometries arising from point configurations in $\mathbb{R}^d$ can be optimized in polynomial time.
This result, as well as the result of Hammer and Kogan~\cite{hammer1995quasi} on acyclic Horn functions, also known as acyclic convex geometries, poset-type convex geometries or $G$-geometries (see, e.g.,~\cite{hammer1995quasi,wild1994theory,adaricheva2017optimum,defrain2021translating}), has then been generalized by Adaricheva~\cite{adaricheva2017optimum}.
In that contribution, the author prove that convex geometries with the Carousel property, $D$-geometries and double-shelling convex geometries do have tractable optimization.
Convex geometries with the Carousel property generalize affine convex geometries.
$D$-geometries are precisely the intersection of lower-bounded and lower-semimodular closure systems, and properly contain acyclic convex geometries.
Double-shelling convex geometries, also well-studied (see, e.g., \cite{edelman1985theory,adaricheva2017optimum,kashiwabara2010characterizations}), are the convex geometries that arise from taking the convex subsets of posets.
Adaricheva leaves open the question of whether convex geometries can be optimized in polynomial time.
Bichoupan~\cite[unpublished]{bichoupan2022complexity} recently answered negatively this question, proving that convex geometries are hard to optimize. 
Understanding what makes optimization tractable or not in convex geometries thus constitutes an interesting research question.

Our contribution more precisely follows this line.
By associating to each essential set $C$ (see Section~\ref{sec:preliminaries} for definitions) of a given convex geometry $(\U, \cs)$ a hypergraph $\HQC(C)$, which we call the \emph{quasi-closed hypergraph} of $C$, whose edges are the complement of the inclusion-wise maximal quasi-closed sets of $C$, we first propose a characterization of all the (left-)optimum bases of the convex geometry:
\begin{theorem}[restate=THMcgoptimib, label=thm:cg-optim-ib]
Let $(\U, \cs)$ be a convex geometry.
A pair $(\U, \ib)$ where $\ib$ is a collection of valid implications of $(\U, \cs)$ is a left-optimum (resp. optimum) implicational base of $(\U, \cs)$ if and only if $\ib$ is obtained by choosing, for each essential set $C$, exactly one implication of the form $\ex(C) \imp T$ where $T$ is a hitting set (resp.\! minimum hitting set) of $\HQC(C)$.
\end{theorem}
We then focus on the case where all the quasi-closed hypergraphs have pairwise disjoint edges.
In this case, minimal and minimum hitting sets of these hypergraphs coincide.
Therefore, based on Theorem~\ref{thm:cg-optim-ib}, any greedy reduction of the conclusions of a left-optimum implicational base will produce an optimum implicational base.
Since in convex geometries each closed set has a unique minimal spanning set, left-optimality is easily reached by a greedy reduction of the premises.
Hence, any polynomial time algorithms such as in \cite{maier1980minimum,shock1986computing} that minimize and reduce an implicational base will in this case produce an optimum base, which yields our second theorem:
\begin{theorem}[restate=THMqcdisjoint, label=thm:qc-disjoint]
An implicational base of a convex geometry where each quasi-closed hypergraph has pairwise disjoint edges can be optimized in polynomial time.
\end{theorem}

We finally prove that Theorem~\ref{thm:qc-disjoint} applies to all the classes of convex geometries for which, up to our knowledge, optimization has been shown to be tractable: double shelling convex geometries, $D$-geometries (including acyclic convex geometries) and convex geometries with the Carousel property (including affine convex geometries).
We also show that it applies to acceptant convex geometries, a class that has been studied within the context of choice functions and stable matchings~\cite{alkan2001preferences,chambers2017choice,chambers2018lexicographic,dougan2021capacity} where convex geometries are known as path-independent choice functions.
More precisely, we prove the correctness of the hierarchy depicted in Figure~\ref{fig:hierarchy} that also summarizes the known results regarding optimization of convex geometries.
Our main claim in this regard is the following theorem:

\begin{theorem}[restate=THMcgclasses, label=thm:cg-classes]
The following statements hold:
\begin{enumerate}[label=(\arabic*)]
    \item in a double-shelling convex geometry, each quasi-closed hypergraph has disjoint edges;
    \item in a $D$-geometry, each quasi-closed hypergraph has singleton edges;
    \item in a convex geometry with the Carousel property, each quasi-closed hypergraph has either a unique edge or singleton edges;
    \item in an acceptant convex geometry, each quasi-closed hypergraph has a unique edge.
\end{enumerate}
A convex geometry in any of these cases thus has quasi-closed hypergraphs with disjoint edges and hence any of its implicational base can be optimized in polynomial time.
\end{theorem}

We then propose a series of examples showing that the proper inclusions given in Figure~\ref{fig:hierarchy} are indeed proper and that no other inclusions hold.
Note that Adaricheva~\cite{adaricheva2017optimum} already proved that $D$-geometries properly contain acyclic convex geometries and that convex geometries with the Carousel property properly include affine convex geometries.

\begin{figure}[t]
    \centering
    \includegraphics[page=1, width=\linewidth]{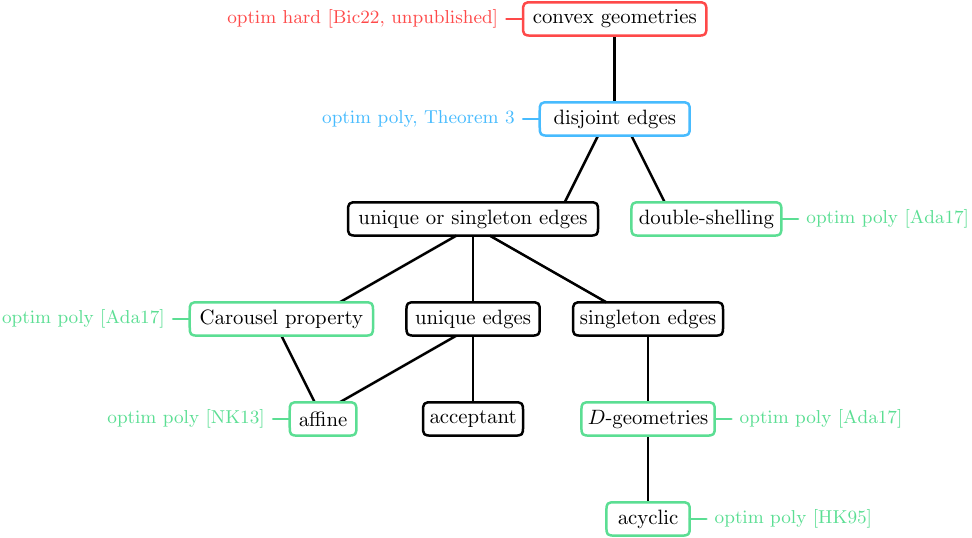}
    \caption{The hierarchy of classes and properties of convex geometries we consider in this paper (from bottom to top). 
    An edge indicates proper inclusion.
    Disjoint edges (resp.\! unique, singleton, and unique or singleton edges) stands for convex geometries where each quasi-closed hypergraph has disjoint edges (resp.\! a unique edge, singleton edges, and a unique or singleton edges).
    The complexity of the optimization task for a given class/property is indicated, with the corresponding reference, next to the associated box.
    For instance, optimization is shown to be polynomial for double-shelling convex geometries in \cite{adaricheva2017optimum}.
    }
    \label{fig:hierarchy}
\end{figure}

\paragraph{Paper Organization.}
We recall some definitions and give some notations in Section~\ref{sec:preliminaries}.
In Section~\ref{sec:optim-IB-CG} we define quasi-closed hypergraphs and prove Theorem~\ref{thm:cg-optim-ib}.
In Section~\ref{sec:Q-disjoint-CG} we prove Theorems~\ref{thm:qc-disjoint} and \ref{thm:optim-classes}.
Each class of Theorem~\ref{thm:optim-classes} is discussed in a separate subsection: double-shelling convex geometries in Subsection~\ref{subsec:posets-cg}, $D$-geometries in Subsection~\ref{subsec:D-geometries}, convex geometries with the Carousel property in Subsection~\ref{subsec:carousel} and acceptant convex geometries in Subsection~\ref{subsec:q-acc-cg}.
The examples showing that the strict inclusions in Figure~\ref{fig:hierarchy} and that no other inclusion holds are given in Appendix~\ref{subsec:hierarchy}.
We conclude the paper in Section~\ref{sec:discussion} with some discussions for further research.

\section{Preliminaries} \label{sec:preliminaries}

All the object we consider in this paper are finite.
Let $\U$ be a set.
We denote by $\pow{\U}$ its powerset. By $|\U|$ its size.
If $\mathcal{H}$ is a hypergraph, a \emph{hitting set} of $\mathcal{H}$ is set of vertices that intersects every edge of $\H$.
A hitting set $T$ is \emph{minimal} if it is inclusion-wise minimal among hitting sets, and \emph{minimum} if it is cardinality minimum among hitting sets.
Sometimes we write a set as the concatenation of its elements, i.e., $abc$ instead of $\{a, b, c\}$.
We also withdraw brackets for singletons when it is clear from the context.
We proceed to recall definitions on closure systems, implicational bases, convex geometries.

\paragraph{Closure operator, closure systems.}
A \emph{closure operator} over a ground set $\U$ is a map $\cl : \pow{\U} \to \pow{\U}$ that satisfies, for all $A, B \subseteq \U$: (1) $A \subseteq \cl(A)$, (2) $A \subseteq B$ entails $\cl(A) \subseteq \cl(B)$ and (3) $\cl(\cl(A)) = \cl(A)$.
A set $C \subseteq \U$ is \emph{closed} w.r.t. $\cl$ if $\cl(C) = C$.
The pair $(\U, \cs)$ where $\cs = \{C \st C \subseteq \U \text{ and } \cl(C) = C\}$ is a \emph{closure system}.
A closure system is a pair $(\U, \cs)$ where $\cs$ is a collection of subsets of $\U$, called \emph{closed sets}, such that $\U \in \cs$ and $C_1 \cap C_2 \in \cs$ for every $C_1, C_2 \in \cs$.
Each closure system $(\U, \cs)$ has an associated closure operator $\cl$ over $\U$ defined by $\cl(A) = \bigcap \{C \st C \in \cs, A \subseteq C\}$ for all $A \subseteq \U$.
Therefore, there is a one-to-one correspondence between closure operator and closure systems.
The closure operator $\cl$ of a closure system induces a partition of $\pow{\U}$ into equivalence classes, where the equivalence class of a set $A$ is the family $[A]_\cl = \{B \st B \subseteq \U, \cl(B) = \cl(A)\}$.
If $(\U, \cs)$ is a closure system, the poset $(\cs, \subseteq)$ of closed sets ordered by inclusion form a \emph{(closure) lattice} where, for any $C_1, C_2 \in \cs$, $C_1 \land C_2 = C_1 \cap C_2$ and $C_1 \lor C_2 = \cl(C_1 \cup C_2)$.
Let $C_1 \subseteq C_2$ be two closed sets.
We say that $C_1$ is a \emph{predecessor} of $C_2$ in $(\cs, \subseteq)$, equivalently that $C_2$ \emph{covers} $C_1$, if there is no $C_3 \in \cs$ such that $C_1 \subset C_3 \subset C_2$.

Let $(\U, \cs)$ be a closure system and let $A \subseteq \U$.
If $C \in \cs$, we say that $A$ is a \emph{spanning set} of $C$ if $\cl(A) = C$.
For $x \in \U$ with $x \notin A$, we say that $A$ is a \emph{(non-trivial) generator} of $x$ if $x \in \cl(A)$.
An \emph{extreme point} of $A$ is an element $x \in A$ such that $x \notin \cl(A \setminus \{x\})$.
We denote by $\ex(A)$ the extreme points of $A$.
Note that for any $C \in \cs$, $\ex(C)$ is included in any (inclusion-wise) minimal spanning set of $C$.
Let $Q \subseteq \U$.
We say that $Q$ is \emph{quasi-closed} if it is non-closed and for every $A \subseteq Q$, $\cl(A) \subset \cl(Q)$ entails $\cl(A) \subseteq Q$.\footnote{Closed sets are often considered quasi-closed, but it will be more convenient for us to assume they are not.}
A set $P$ is \emph{pseudo-closed} if it is an inclusion-wise minimal quasi-closed set among the spanning sets of $\cl(P)$.
A closed set which is the closure of a quasi-closed (or pseudo-closed) set is \emph{essential}.
Note that for any closed set $C$, the quasi-closed sets of the closure system $(C, \{C' \st C' \in \cs, C' \subseteq C\})$ are exactly the quasi-closed sets of $(\U, \cs)$ included in $C$.
This holds in particular for pseudo-closed sets.
Quasi-closed sets are also related to the \emph{saturation operator} $\sigma$ of $(\U, \cs)$.
This operator is defined for all $Y \subseteq \U$ by $\sigma(Y)=\bigcup_{k\geq 1} \psi^k(Y)$, where $\psi(Y) = Y \cup \bigcup\{\phi(Z)\st Z\subseteq Y \text{ and } \phi(Z)\subset \phi(Y)$\}. 
The operator $\sigma$ is a closure operator whose closed sets are quasi-closed and closed sets of $(\U, \cs)$.
In other words, $Q$ is quasi-closed if and only if $Q$ is not closed and $\sigma(Q) = Q$.

We finally define convex geometries.
A closure system $(\U, \cs)$ is a \emph{convex geometry} if $\emptyset \in \cs$ and for each closed set $C \subset \U$, there exists $x \notin C$ such that $C \cup \{x\}$ is closed.
It is well-known that a closure system $(\U, \cs)$ is a convex geometry if and only if its closure operator has the \emph{antiexchange} property, that is, for any $Y \subseteq \U$ and any $a, b \notin \cl(Y)$, $a \in \cl(Y \cup \{b\})$ implies that $b \notin \cl(Y \cup \{a\})$.
Convex geometries can also be characterized by their spanning sets: $(\U, \cs)$ is a convex geometry if and only if every closed set $C$ has a unique minimal spanning set being $\ex(C)$.
As a consequence, we observe that the predecessors of each closed set $C \neq \emptyset$ are precisely the sets $C \setminus \{x\}$ for $x \in \ex(C)$. 

\paragraph{Implicational bases.}
We turn our attention to implications.
We redirect the reader to~\cite{bertet2018lattices,wild2017joy} for a detailed introduction to the topic.
An \emph{implication} over $\U$ is an expression $A \imp B$ where $A, B$ are subsets of $\U$.
The set $A$ is called the \emph{premise} and $B$ the \emph{conclusion} of the implication.
Given a collection $\ib$ of implications over $\U$,
we call the pair $(\U, \ib)$ an \emph{implicational base}.
The number of implications in $\ib$ is denoted by $\card{\ib}$.
The \emph{left-size} of $\ib$ is defined as $\leftsize{\ib} := \sum_{A \imp B \in \ib} \card{A}$.
Similarly, $\rightsize{\ib} := \sum_{A \imp B \in \ib} \card{B} $ denotes the \emph{right-size} of $\ib$.
The \emph{total size} of $\ib$ is defined as $\totalsize{\ib} := \sum_{A \imp B \in \ib} \card{A} + \card{B}$. 
Note that we equivalently have $\totalsize{\ib} = \leftsize{\ib} + \rightsize{\ib}$.

We relate implications to closure systems.
Let $A \imp B$ be an implication and $C \subseteq \U$.
We say that $C$ \emph{satisfies} $A \imp B$ if $A \subseteq C$ implies $B \subseteq C$.
Let $(\U, \ib)$ be an implicational base.
The set $C$ satisfies $\ib$ if it satisfies all the implications of $\ib$.
The pair $(\U, \cs)$ where $\cs$ is the collection of subsets of $\U$ satisfying $\ib$ is known to form a closure system.
Its associated closure operator $\ib$ can be computed in polynomial time from the pair $(\U, \ib)$ using the forward chaining procedure.
For each $A \subseteq \U$, this procedure builds a sequence $A = Y_0 \subseteq Y_1 \subseteq \dots \subseteq Y_k = \ib(A)$ of sets defined by $Y_i = Y_{i-1} \cup \bigcup \{B \st A \imp B \in \ib, A \subseteq Y_{i-1}\}$.
An \emph{equivalence class} of $(\U, \ib)$ is a set of implications of $\ib$ whose premises have the same closure.
Any implicational base $(\U, \ib)$ is thus partitioned by its equivalence classes.
Let $(\U, \cs)$ be a closure system.
An implication $A \imp B$ over $\U$ is \emph{valid} in $(\U, \cs)$ if every closed set satisfies the implication.
Note that $A \imp B$ is valid in $(\U, \cs)$ if and only if $B \subseteq \cl(A)$.
An implicational base is a \emph{(valid) implicational base} of $(\U, \cs)$ if the closure system associated to $(\U, \ib)$ is precisely $(\U, \cs)$.
In general, an implicational base consisting only of valid implications of $(\U, \cs)$ does not need to be an implicational base of $(\U, \cs)$.
A closure system can usually be represented by several different implicational bases.
Two implicational bases with the same closure system are \emph{equivalent}.
Let $(\U, \ib)$ be an implicational base of $(\U , \cs)$.
We say that $(\U, \ib)$ is \emph{minimum} if $\card{\ib} = \min \{\card{\ib'} \st (\U, \ib') \text{ is an implicational base of } (\U, \cs)\}$.
Similarly, we say that $(\U, \ib)$ is \emph{optimum} (resp. \emph{left-optimum}, \emph{right-optimum}) if $\ib$ minimizes $\totalsize{\cdot}$ (resp. $\leftsize{\cdot}$, $\rightsize{\cdot}$)
among the implicational bases of $(\U, \cs)$.
We say that $(\U, \ib)$ is \emph{left-reduced} if for every $A \imp B \in \ib$ and every $A' \subset A$, $(\ib \setminus \{A \imp B\}) \cup \{A' \imp B\}$ is no longer an implicational base of $(\U, \cs)$.
Similarly, $(\U, \ib)$ is \emph{right-reduced} if $(\ib \setminus \{A \imp B\}) \cup \{A \imp B'\}$ is no longer an implicational base of $(\U, \cs)$ for every $B' \subset B$ and every $A \imp B \in \ib$.
Finally, $(\U, \ib)$ is \emph{reduced} if it is both left- and right-reduced.
Below we recall known results regarding optimum implicational bases.
Notably, we recall that optimality can be achieved by optimizing each equivalence class of the implicational base independently from the others.
In other words, the total sizes of the implicational bases of each equivalence class in an optimum base is fixed.
\begin{theorem}[see, e.g., \cite{ausiello1986minimum,maier1983theory,bichoupan2022complexity,mannila1983relationship}]\label{thm:optim-classes}
Let $(\U, \cs)$ be a closure system and let $(\U, \ib)$ be an implicational base of $(\U, \cs)$.
Then:
\begin{enumerate}[label=(\arabic*)]
    \item if $(\U, \ib)$ is (left-)optimum, it is also minimum;
    \item $(\U, \ib)$ is optimum if and only if it is both left- and right-optimum;
    \item $(\U, \ib)$ is optimum if and only if all its equivalence classes are optimum.
\end{enumerate}
\end{theorem}

The \emph{canonical base} of $(\U, \cs)$ is the implicational base $(\U, \ib_c)$ where:
\[ \ib_c := \{P \imp \cl(P) \setminus P \st P \text{ is pseudo-closed } \} \]

\begin{theorem}[see, e.g., \cite{ganter1986two,guigues1986familles,wild1994theory}] \label{thm:classic-pseudo}
Let $(\U, \cs)$ be a closure system and let $(\U, \ib_c)$ be its canonical base.
Then:
\begin{enumerate}[label=(\arabic*)]
    \item $(\U, \ib_c)$ is a minimum base of $(\U, \cs)$ and every other implicational base $(\U, \ib)$ of $(\U, \cs)$ contains for each pseudo-closed set $P$, an implication $A \imp B$ such that $A \subseteq P$ and $\cl(A) = \cl(P)$.
    \item for each left-optimum base $(\U, \ib)$ of $(\U, \cs)$ and for each pseudo-closed set $P$, $\ib$ contains an implication $A \imp B$ such that $A \subseteq P$, $\cl(A) = \cl(P)$ and $\card{A} = \min \{ \card{Y} \st Y \subseteq P, \cl(Y) = \cl(P)\}$.
\end{enumerate}
\end{theorem}

Finally, we recall that the saturation operator of a closure system can be computed from any of its implicational bases in polynomial time.
\begin{lemma} \cite[Lemma 4]{wild1994theory} \label{lem:wild-saturation-poly}
Let $(\U, \ib)$ be an implicational base of $(\U, \cs)$ and let $Y \subseteq \U$.
Then, $\sigma(Y) = \ib_Y(Y)$ where $\ib_Y = \ib \setminus \{A \imp B \st A \imp B \in \ib, \cl(A) = \cl(Y)\}$.
\end{lemma}

\section{Quasi-closed hypergraphs and optimum bases of convex geometries} \label{sec:optim-IB-CG}

In this section we give a characterization of the optimum bases of a convex geometry.
This characterization relies on the minimum hitting sets of the hypergraphs of the quasi-closed sets that are, within a same equivalence class, inclusion-wise maximal.

The starting point of our approach lies in the more general context of arbitrary closure systems.
Consider some closure system  $(\U, \cs)$ and let $\ib$ be a collection of valid implications of $(\U, \cs)$.
In the next proposition, we observe that what makes $(\U, \ib)$ an implicational base of $(\U, \cs)$ or not does only depend on its ability to make every quasi-closed set indeed non-closed using implications with premises from its equivalence class.
Note that this has also been observed by Bichoupan~\cite[unpublished]{bichoupan2023independence}, although we propose a slightly more precise version.
We still give a proof of this claim for self-completeness.

\begin{proposition}[Lemma 2 in \cite{bichoupan2023independence}] \label{prop:carac-valid-qc}
Let $\ib$ be a collection of valid implications of $(\U, \cs)$.
The pair $(\U, \ib)$ is an implicational base of $(\U, \cs)$ if and only if for every quasi-closed set $Q$ of $\cs$, there exists $A \imp B \in \ib$ such that $A \subseteq Q$, $\cl(A) = \cl(Q)$, and $B$ contains an element of $\phi(Q) \setminus Q$.
\end{proposition}

\begin{proof}
We start with the only if part.
Assume $(\U, \ib)$ is an implicational base of $(\U, \cs)$ and let $Q$ be a quasi-closed set of $\cs$.
Since $(\U, \ib)$ is an implicational base of $(\U, \cs)$ and $Q$ is not closed, there exists $A \imp B \in \ib$ such that $A \subseteq Q$ but $B \nsubseteq Q$.
As $Q$ is quasi-closed, $\cl(A) \subset \cl(Q)$ would entail $B \subseteq \cl(A) \subset Q$, a contradiction with the fact that $B \nsubseteq Q$.
We deduce that $\cl(A) = \cl(Q)$ and that $B$ contains an element of $\cl(Q) \setminus Q$.

We move to the if part.
It is sufficient to show that $\cl(Q) = \ib(Q)$ for each quasi-closed set $Q$.
Since $(\U, \ib)$ contains only valid implications of $(\U, \cs)$, we readily have $\ib(Q) \subseteq \cl(Q)$.
Hence, it remains only to argue that $\cl(Q) \subseteq \ib(Q)$ for each quasi-closed set $Q$.
Assume for contradiction this is not the case and let $C$ be an inclusion-wise minimal closed set of $\cs$ where such a $Q$ exists.
We thus have $\ib(Q) \subset \cl(Q)$.
By assumption on $\ib$, $\ib(Q)$ cannot be quasi-closed in $(\U, \cs)$ as otherwise there would exist an implication $A \imp B$ in $\ib$ such that $A \subseteq \ib(Q)$, $A$ spans $\cl(Q)$ and $B \nsubseteq \ib(Q)$.
Therefore, there exists $Y \subseteq \ib(Q)$ such that $\cl(Y) \subset \cl(\ib(Q)) = C$ and $\cl(Y) \nsubseteq \ib(Q)$.
However, by minimality assumption on $C$, for every $C' \in \cs$ such that $C' \subset C$ and any quasi-closed set $Q'$ spanning $C'$, $\ib(Q') = \cl(Q')$.
This holds in particular for all pseudo-closed sets $P$ such that $\cl(P) \subset C$.
Therefore, $\cl(Y) = \ib(Y)$ must hold, and $\ib(Y) \subseteq \ib(Q)$ too as $Y \subseteq \ib(Q)$.
This contradicts $\cl(Y) \nsubseteq \ib(Q)$, and concludes the proof.
\end{proof}

\begin{example} \label{ex:carac-valid-qc}
Let $\U = \{a, \dots, f\}$ and consider the closure system $(\U, \cs)$ depicted in Figure~\ref{fig:carac-valid-qc}.
It has three essential sets being $\lset{ab}$, $\lset{cdef}$, $\lset{abcdef}$ and its canonical base $(\U, \ib_c)$ is given by:
\[ \ib_c = \{a \imp b, b \imp a, f \imp \lset{cde}, \lset{ce} \imp \lset{df}, \lset{de} \imp \lset{cf}, \lset{abe} \imp \lset{cdf}, \lset{abd} \imp \lset{cef} \}
\] 
Consider for instance the quasi-closed sets of $\lset{cdef}$.
They coincide with its non-closed spanning sets: $\lset{ce}$, $\lset{de}$, $\lset{f}$ and every proper subset of $\lset{cdef}$ that includes at least one of these three sets.
By Proposition~\ref{prop:carac-valid-qc}, each quasi-closed set of $\lset{cdef}$ is non-closed in $(\U, \ib_c)$ due to an implication with a premise spanning $\lset{cdef}$.
For example, the quasi-closed sets $\lset{cf}$ and $\lset{ef}$ relate to the implication $f\imp \lset{cde}$.
\begin{figure}[ht!]
    \centering
    \includegraphics[width=\linewidth]{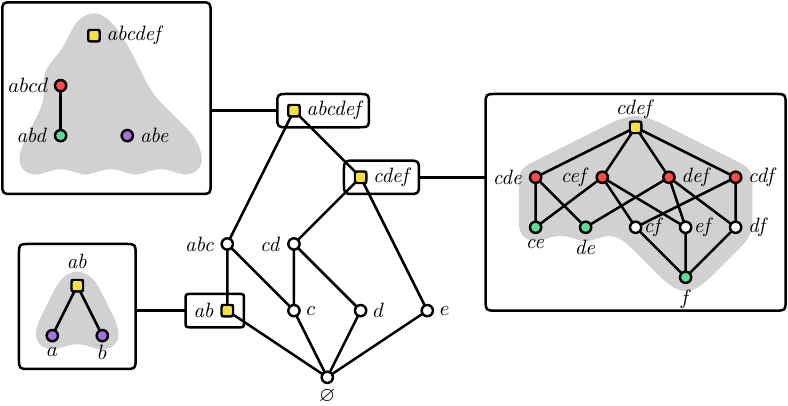}
    \caption{The closure system of Example~\ref{ex:carac-valid-qc}. It has three essential sets (yellow square nodes) being $\lset{ab}$, $\lset{cdef}$, $\lset{abcdef}$. 
    To each essential set is associated a box representing its quasi-closed sets (circle nodes) ordered by inclusion within its spanning sets (shaded zone).
    For instance the quasi-closed sets of $\lset{abcdef}$ are $\lset{abd}$, $\lset{abe}$, $\lset{abcd}$.
    Minimal quasi-closed sets are green (e.g., $\lset{abd}$), maximal ones red (e.g., $\lset{abcd}$), and those that are both minimal and maximal purple (e.g., $\lset{abe}$).}
    \label{fig:carac-valid-qc}
\end{figure}
\end{example}

Suppose that $(\U, \ib)$ is indeed an implicational base of $(\U, \cs)$ and let $C$ be an essential set of $\cs$.
From Proposition~\ref{prop:carac-valid-qc}, we obtain that $T := \bigcup \{B \st A \imp B \in \ib \text{ and } \cl(A) = C\}$ is a hitting set of the hypergraph with vertex set $C$ and edges the complements of the quasi-closed sets of $C$. 
Since some quasi-closed sets may be comparable, we can more simply observe that $T$ is a hitting set of the hypergraph $\HQC(C)$, which we call the \emph{quasi-closed hypergraph of $C$}, having vertex set $C$ and edge set:
\[ \{ C \setminus Q \st Q \text{ is an inclusion-wise maximal quasi-closed set of } C\}. \]

\begin{example}[continued] \label{ex:qc-hyp}
The maximal quasi-closed sets of $\lset{cdef}$ are $\lset{cde}$, $\lset{cef}$, $\lset{def}$, $\lset{cdf}$, so that $\HQC(\lset{cdef}) = \{f, d, c, e\}$.
Similarly, the maximal quasi-closed set of $\lset{abcdef}$ are $\lset{abcd}$, $\lset{abe}$.
We thus have $\HQC(\lset{abcdef}) = \{\lset{ef}, \lset{cdf}\}$.
\end{example}

Proposition~\ref{prop:carac-valid-qc} now reads as follows: to build implicational bases of $(\U, \cs)$, one has to select for each essential set $C$ enough premises to subsume all quasi-closed sets of $C$ and then a hitting set $T$ of $\HQC(C)$ to distribute over these premises.
This applies in particular to minimum implicational base of $(\U, \cs)$ that can be obtained by applying the following steps to each essential set $C$:
\begin{enumerate}
    \item choose a (unique) premise $A \subseteq P$ for each pseudo-closed set $P$ of $C$;
    \item choose a hitting set $T$ of $\HQC(C)$;
    \item distribute $T$ over the premises $A$ so as to fulfill Proposition~\ref{prop:carac-valid-qc}.
    Note that $T$ needs not be partitioned over the premises.
\end{enumerate}
The resulting $(\U, \ib)$ is indeed a minimum implicational base of $(\U, \cs)$ due to Proposition~\ref{prop:carac-valid-qc} and Theorem~\ref{thm:classic-pseudo}.
Besides, we indeed obtain all possible minimum implicational bases follows from as we choose for each essential set $C$ and each pseudo-closed set of $C$ all possible premises $A$ and all possible hitting set $T$ of $\HQC(C)$, where $T$ can be any subset of $C$, that we distribute in all valid ways.
In particular, the canonical base is built by taking each pseudo-closed set $P$ as premise and repeating the trivial hitting set $C$ of $\HQC(C)$ in each conclusions, possibly withdrawing useless elements.

\begin{example}[continued] \label{ex:build-canonical}
We illustrate the construction of $(\U, \ib_c)$ in Example \ref{ex:carac-valid-qc} by considering each essential set:
\begin{itemize}
    \item $\HQC(\lset{ab}) = \{a, b\}$ with hitting set $ab$ and pseudo-closed sets $a$, $b$. The implications thus are $a \imp b$ and $b \imp a$ ;
    \item $\HQC(\lset{cdef}) = \{f, d, c, e\}$ with hitting set $\lset{cdef}$ and pseudo-closed sets $\lset{ce}$, $\lset{de}$, $\lset{f}$.
    We obtain implications $f \imp \lset{cde}$, $\lset{ce} \imp \lset{df}$, $\lset{de} \imp \lset{cf}$ ;
    \item $\HQC(\lset{abcdef}) = \{\lset{ef}, \lset{cdf}\}$ with hitting set $\lset{abcdef}$.
    The pseudo-closed sets being $\lset{abe}$, $\lset{abd}$, we obtain $\lset{abe} \imp \lset{cdf}$, $\lset{abd} \imp \lset{cef}$.
\end{itemize}
\end{example}

Let us restrict our attention to optimum bases, that are minimum by Theorem~\ref{thm:optim-classes}.
By Theorem~\ref{thm:optim-classes}, we can optimize the equivalence class of each essential set $C$ independently from the others, and optimality is found by finding left- and right-optimality:
\begin{enumerate}
    \item left-optimality is obtained by selecting for each pseudo-closed set $P$ a minimal spanning set $A \subseteq P$ of $\cl(P)$ with minimum cardinality (item (2) in Theorem~\ref{thm:classic-pseudo}) ;
    \item right-optimality is obtained by selecting a hitting set $T$ of $\HQC(C)$ that can be distributed over the premises of $C$ in a minimum way.
\end{enumerate}
For right-optimality in general, a given element of $T$ might need to be repeated in conclusion of different implications.
As a consequence, $T$ does not need to be a minimum hitting set of $\HQC(C)$ in general as we illustrate below:

\begin{example}[continued] \label{ex:build-optim}
We construct an optimal implicational base $(\U, \ib)$ for the closure system of Example~\ref{ex:carac-valid-qc} by looking at each essential set.
For $\lset{ab}$, the unique possible choice is $a \imp b$ and $b \imp a$.
For $\lset{cdef}$, the left-optimum premises are its minimal spanning sets, $\lset{f}$, $\lset{ce}$, $\lset{de}$.
As for conclusions, $\HQC(C)$ has a unique hitting set being $\lset{cdef}$.
However, the distribution of $\lset{cdef}$ may produce right-reduced implications that are not optimum.
For instance $f \imp \lset{cde}$, $\lset{ce} \imp f$, $\lset{de} \imp f$ cannot be reduced but is not optimum as we can reach 4 conclusions only with: $f \imp \lset{ce}$, $\lset{ce} \imp d$, $\lset{de} \imp f$.
Finally for $\lset{abcdef}$, optimum premises are for instance $ad$, $ae$.
As for conclusions, $\HQC(\lset{abcdef}) = \{\lset{ef}, \lset{cdf}\}$ has hitting sets, e.g., $f$ and $\lset{ce}$ which are of different size.
Yet, since two implications are needed, both hitting sets produce optimum implications: one can take either $\{ad \imp f, ae \imp f\}$ or $\{ad \imp e, ae \imp c\}$.
An example of optimum base is thus given by $\ib = \{a \imp b, b \imp a, f \imp \lset{ce}, \lset{ce} \imp d, \lset{de} \imp f, ad \imp e, ae \imp c\}$.
\end{example}

The situation simplifies for convex geometries.
Let us then assume that $(\U, \cs)$ is a convex geometry and let $C$ be some essential set.
Since $C$ has a unique minimal spanning set, $\ex(C)$, $C$ plus the quasi-closed sets of $C$ form a closure system.
In particular, it follows that $C$ has a unique pseudo-closed set and that each minimum implicational base of $(\U, \cs)$ has a unique implication $A \imp T$ for the equivalence class of $C$.
This makes $T$ on its own a hitting set of $\HQC(C)$ by Proposition~\ref{prop:carac-valid-qc}.
Moreover, item (2) of Theorem~\ref{thm:classic-pseudo} entails that $\ex(C)$ is the unique choice of minimal spanning set for any left-optimum implicational base of $(\U, \cs)$ (see also, e.g., \cite[Corollary 13(b)]{wild1994theory}, \cite[Theorem 7]{adaricheva2017optimum}). 
We obtain the first part of Theorem~\ref{thm:cg-optim-ib} regarding left-optimality that we prove in a separate lemma:

\begin{lemma} \label{lem:cg-left-optim-ib}
Let $(\U, \cs)$ be a convex geometry.
A pair $(\U, \ib)$ where $\ib$ is a collection of valid implications of $(\U, \cs)$ is a left-optimum implicational base of $(\U, \cs)$ if and only if $\ib$ is obtained by choosing, for each essential set $C$, exactly one implication of the form $\ex(C) \imp T$ where $T$ is a hitting set of $\HQC(C)$.
\end{lemma}

\begin{proof}
We first show the if part.
For simplicity, let $C_1, \dots, C_m$ be the essential sets of $(\U, \cs)$ and, for each $1 \leq i \leq m$, let $T_i$ be some hitting set of $\HQC(C_i)$.
Let $\ib = \{\ex(C_1) \imp T_1, \dots, \ex(C_m) \imp T_m\}$.
We prove that $(\U, \ib)$ is an implicational base of $(\U, \cs)$.
Observe that $\ib$ contains only valid implications of $(\U, \cs)$, i.e., $T_i\subseteq C_i$.
Let $Q$ be a quasi-closed set with closure $C_i$, for some $i$.
Since $(\U, \cs)$ is a convex geometry, $\ex(C_i) \subseteq Q$.
Moreover, $T_i$ contains an element of $C_i \setminus Q$ as $T_i$ is a hitting set of $\HQC(Q_i)$ by definition.
Applying this reasoning to any quasi-closed sets yields that $(\U, \ib)$ is indeed an implicational base of $(\U, \cs)$ by Proposition~\ref{prop:carac-valid-qc}.
The fact that it is left-optimal follows from items (2) of Theorem~\ref{thm:classic-pseudo} and the fact that $(\U, \cs)$ is a convex geometry.

We now prove the only if part.
Let $(\U, \ib)$ be a left-optimum base of $(\U, \cs)$.
Again using items (2) of Theorem~\ref{thm:classic-pseudo} and the fact that $(\U, \cs)$ is a convex geometry, we have that $\ib = \{\ex(C_1) \imp T_1, \dots, \ex(C_m) \imp T_m\}$ where each $T_i$ is some subset of $C_i$.
As $(\U, \ib)$ is an implicational base of $(\U, \cs)$, we have by Proposition~\ref{prop:carac-valid-qc} that for each maximal closed set $Q$ of $C_i$, $T_i$ contains an element of $C_i \setminus Q$.
It follows that $T_i$ is a hitting set of $\HQC(C_i)$ and that $\ib$ is indeed obtained by choosing exactly one implication of the form $\ex(C) \imp T$ for each essential set $C$ and some hitting set $T$ of $\HQC(C)$.
This concludes the proof.
\end{proof}

Since optimum bases are left-optimum, we deduce that optimum bases of $(\U, \cs)$ falls into the scope of Lemma~\ref{lem:cg-left-optim-ib}.
In particular, right-optimality is reached exactly when for each essential set $C$, the size of the chosen hitting set $T$ of $\HQC(C)$ is minimum.
We thus obtain Theorem~\ref{thm:cg-optim-ib}:
\THMcgoptimib*

For convex geometries, Theorem~\ref{thm:cg-optim-ib} shows that optimization amounts to identify a minimum hitting set $T$ of $\HQC(C)$ of each essential set $C$, possibly by explicitly computing $\HQC(C)$, and then outputting $\ex(C) \imp T$. 
This strategy will not produce polynomial time algorithms in general as optimizing convex geometries is $\NP$-hard~\cite[unpublished]{bichoupan2022complexity}.
To show hardness, the author implicitly proves that any closure system can be turned into the closure system of quasi-closed sets of an essential set $C$ of a convex geometry $(\U, \cs)$.
Since finding a minimum hitting set $T$ of $\HQC(C)$ equivalently reads as finding a minimum spanning set of $C$ in the closure system of the quasi-closed sets of $C$ being given by an implicational base, which is a hard problem~\cite{lucchesi1978candidate}, hardness of optimizing convex geometries follows.
In the next section we use Theorem~\ref{thm:cg-optim-ib} to identify a property of convex geometries that makes optimization tractable and show this property applies to different classes of convex geometries. 

\section{Quasi-closed hypergraphs with disjoint edges} \label{sec:Q-disjoint-CG}

Drawing from Theorem~\ref{thm:cg-optim-ib}, we investigate the case where each quasi-closed hypergraph of a convex geometry $(\U, \cs)$ has pairwise disjoint edges.
We prove that any implicational base of $(\U, \cs)$ can be optimized in polynomial time (  Theorem~\ref{thm:qc-disjoint}), and then that several classes of convex geometries do enjoy this property (Theorem~\ref{thm:cg-classes}).

For Theorem~\ref{thm:qc-disjoint} we simply argue that, under our assumptions, any polynomial time algorithm that takes as input an implicational base $(\U, \ib)$ of $(\U, \cs)$ and turns it into an equivalent minimum reduced base does in fact make it optimum.
Minimization and reduction can be achieved in polynomial time by applying, say, Shock's algorithm for minimization~\cite{shock1986computing} and then Maier's algorithm for reduction~\cite[Algorithm 5.7, p. 77]{maier1983theory} which first applies left-reduction and then right-reduction.
Informally, assuming $(\U, \ib)$ is already minimum and considering some implication $A \imp T$, left-reduction will turn $A$ into $\ex(C)$ where $C$ is the essential set spanned by $A$ while right-reduction will turn $T$ into a minimal hitting set of $\HQC(C)$ by Theorem~\ref{thm:cg-optim-ib}.
Given that the edges of $\HQC(C)$ are disjoint this minimal hitting set is also minimum, which guarantees optimality again by Theorem~\ref{thm:cg-optim-ib}.

\THMqcdisjoint*
\begin{proof}
Without loss of generality, let $(\U, \ib)$ be a minimum implicational base $(\U, \cs$).
Since $(\U, \cs)$ is a convex geometry, each closed set has a unique minimal spanning set consisting of its extreme points.
Also recall that since $(\U, \cs)$ is a convex geometry and $(\U, \ib)$ a minimum implicational base of $(\U, \cs)$, each implication $A \imp T$ of $\ib$ corresponds to a unique essential set $C$ such that $\cl(A) = C$.
Therefore, applying left-reduction on $(\U, \ib)$ will turn each premise $A$ of $\ib$ into $\ex(C)$ where $C = \cl(A)$ is the essential set spanned by $A$.
Hence, left-reduction turns $(\U, \ib)$ into a left-optimum implicational base of $(\U, \cs)$.
Now we consider right-reduction.
Let $C$ be some essential set with corresponding implication $\ex(C) \imp T$ in $\ib$.
By Theorem~\ref{thm:cg-optim-ib}, right-reducing $\ex(C) \imp T$ will thus produce an implication $\ex(C) \imp T'$ where $T'$ is a hitting set of $\HQC(C)$ but none of its subsets are, i.e., where $T'$ is a minimal hitting set of $\HQC(C)$.
Since the edges of $\HQC(C)$ are disjoint, any of its minimal hitting set is in fact minimum.
Thus, reducing $(\U, \ib)$ does indeed make $(\U, \ib)$ optimum by Theorem~\ref{thm:cg-optim-ib}.
The fact that minimization and reduction can be achieved in polynomial time follows from previous discussion.
This proves the Theorem.
\end{proof}

We now identify four classes of convex geometries to which Theorem~\ref{thm:cg-optim-ib} applies. 
Let us recall that the first three of have been shown to have tractable optimization~\cite{adaricheva2017optimum} and generalize existing subclasses where optimization is polynomial~\cite{hammer1995quasi,nakamura2013prime}.
Our approach thus unifies all these results via the same optimization algorithm.

\THMcgclasses*

Theorem~\ref{thm:cg-classes} thus proves the main inclusion relations of the hierarchy of Figure~\ref{fig:hierarchy}.
Its proof is divided in Subsections~\ref{subsec:posets-cg} to~\ref{subsec:q-acc-cg}, each subsection corresponding to one of the class mentioned in the Theorem.
Examples proving the the correctness of the other relations in Figure~\ref{fig:hierarchy}, proper inclusions and incomparabilities, are delayed to Appendix~\ref{subsec:hierarchy}.
Before proceeding to these subsections, we finally note that in any convex geometry $(\U, \cs)$, for any $x$ such that $\cl(x) \neq \{x\}$, $\HQC(\cl(x))$ has singleton, hence disjoint, edges.

\begin{proposition} \label{prop:join-qc-disjoint}
For any $x \in \U$, if $\cl(x) \neq \{x\}$ then $\HQC(\cl(x)) = \{ \{y\} \st y \in \ex(\cl(x) \setminus \{x\})\}$.
\end{proposition}

\begin{proof}
Let $x$ be such that $\cl(x) \neq \{x\}$.
Since $\emptyset \in \cs$ by definition of convex geometries, $x$ is pseudo-closed ant $\cl(x)$ essential, so that $\HQC(\cl(x))$ is well-defined.
Moreover $x$, is the unique extreme point of $\cl(x)$.
Finally, $\cl(x) \setminus \{x\} \neq \emptyset$ must hold by assumption on $x$.
Let $C := \cl(x) \setminus \{x\}$.
We prove the equality using double-inclusion and start with the $\supseteq$ part.
We show that $\{y\} \in \HQC(\cl(x))$ for every $y \in \ex(C)$, i.e., that $Q := \cl(x) \setminus \{y\}$ is quasi-closed.
But this follows from the fact that for every $A \subseteq Q$ such that $\cl(A) \subset \cl(x)$, we have $A \subseteq \cl(x) \setminus \{x, y\} \subseteq C \setminus \{y\}\in \cs$ with $C \setminus \{y\} \subseteq Q$.
Therefore, $\cl(A) \subseteq Q$ holds.
We deduce $\HQC(\cl(x)) \supseteq \{ \{y\} \st y \in \ex(\cl(x) \setminus \{x\})\}$.

We turn to the $\subseteq$ part, which we show using contradiction.
Given that $\supseteq$ holds, let us thus assume that there exists $Q \in \HQC(\cl(x))$ such that $Q \nsubseteq C \setminus \{y\}$ for each $y \in \ex(C)$.
We have $\ex(C) \subset Q$ and $\cl(C) \subset \cl(x) = \cl(Q)$.
Hence, $C \subseteq Q$ and since $Q$ spans $\cl(x)$, $Q = C$ holds true, which contradicts the fact that $Q$ is non-closed.
This concludes the proof.
\end{proof}

Theorem~\ref{thm:cg-optim-ib} and Proposition~\ref{prop:join-qc-disjoint} directly imply that $\HQC(\cl(x))$ has a unique minimum hitting set and hence that $x \imp \ex(\cl(x) \setminus \{x\})$ belongs to every optimum base of $(\U, \cs)$.
This was also observed by Adaricheva~\cite{adaricheva2017optimum}.

\begin{lemma}[{\cite[Lemma 9]{adaricheva2017optimum}}] \label{lem:join-optim}
For every optimum implicational base $(\U, \ib)$ of $(\U, \cs)$ and every $x \in \U$ such that $\cl(x) \neq \{x\}$, $x \imp \ex(\cl(x) \setminus \{x\}) \in \ib$. 
\end{lemma}


\subsection{Double-shelling convex geometries} \label{subsec:posets-cg}

We prove that any convex geometry in the class of double-shelling convex geometries have quasi-closed hypergraphs with disjoint edges.
Let $P := (\U, \leq)$ be a poset and let $x, y \in \U$.
We write $x < y$ to denote that $x \leq y$ and $x \neq y$. 
If $x < y$, the \emph{interval} of $x, y$ is defined as $[x, y] := \{z \st z \in \U, x \leq z \leq y\}$. 
A set $C \subseteq \U$ is called a \emph{convex set} of $P$ if for any $x, y \in C$ such that $x < y$, $[x, y] \subseteq C$ holds.
Let $\cs$ be the family of convex subsets of $P$.
The pair $\mathrm{Co}(P) := (\U, \cs)$ is known to be a convex geometry called the \emph{double-shelling convex geometry} of $P$.
Its closure operator thus maps any subset $Y$ of $\U$ to the minimal convex set of $P$ that includes $Y$.
Its canonical base $(\U, \ib_c)$ is given by~\cite{adaricheva2017optimum}:
\[ 
\ib_c = \{xy \imp [x, y] \setminus \{x, y\} \st \exists z\in \U, x < z < y \text{ in } P \}.
\] 
Observe that for any essential set $C$ with pseudo-closed set $xy$, we have $\ex(C) = \{x, y\}$ as any singleton is convex.
The \emph{comparability graph} of $P$ is the undirected graph $G_P$ with vertex set $V(G_P) := \U$ and edges $E(G_P) := \{xy \st x < y\}$.
For $x < y$, we define $G_P(x, y) := G_P[[x, y] \setminus \{x, y\}]$ where $G_P[Y]$ is the graph induced by $Y$.
We illustrate these definitions in Example~\ref{ex:poset-cg}.

\begin{example} \label{ex:poset-cg}
Consider the poset $P$ on ground set $\U = \{\lset{a}, \dots, \lset{h},\lset{x},\lset{y}\}$ on the left of Figure~\ref{fig:poset-cg}. 
For instance, $[\lset{a},\lset{y}] = \{a, f, g, y\}$.
The canonical base $(\U, \ib_c)$ of the associated double-shelling convex geometry $\mathrm{Co}(P)$ is given by:
\[ 
\ib_c = \left\{\begin{array}{lllll}
\mathit{xf} \imp \mathit{ab}, & \mathit{xg} \imp \mathit{ab}, & \mathit{xe} \imp c, & \mathit{xh} \imp \mathit{ec}, & \mathit{xy} \imp \mathit{abcdefgh}, \\
\mathit{ay} \imp \mathit{fg}, & \mathit{by} \imp \mathit{fg}, & \mathit{cy} \imp \mathit{eh}, & \mathit{ey} \imp \mathit{h}, & \mathit{ch} \imp e
\end{array}
\right\} 
\] 

\begin{figure}[ht!]
\centering
\includegraphics[page=1, width=\linewidth]{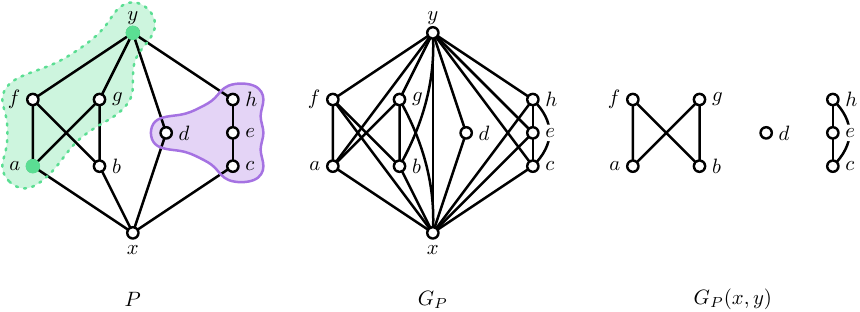}
\caption{On the left the poset $P$ of Example~\ref{ex:poset-cg}.
The set $\mathit{cdeh}$ (circled in purple, solid) is convex.
The interval $[a, y]$ (circled in green, dotted) is $\mathit{afgy}$.
It is also convex.
The comparability graph $G_P$ is pictured in the middle and the graph $G_P(x, y)$ on the right.
It has three connected component: $\mathit{abgf}$, $d$, and $\mathit{ceh}$.}
\label{fig:poset-cg}
\end{figure}
\end{example}

We prove below that for any essential set $C$ with extreme points $x, y$, $\HQC(C)$ consists in the connected components of $G_P(x, y)$ (assuming without loss of generality $x < y$).
While this  can be deduced from \cite[Lemma 26]{adaricheva2017optimum}, we give a proof in our framework based on quasi-closed sets for self-containment.

\begin{lemma} \label{lem:posets-qc}
Let $C$ be an essential set of $\cs$ with $\ex(C) = \{x, y\}$ and $x < y$ in $P$.
Then, the edges of $\HQC(C)$ are the connected components of $G_P(x, y)$, therefore they are disjoint.
\end{lemma}

\begin{proof}
To prove the claim, it is sufficient to show that a set $\{x, y\} \subseteq Q \subseteq C$ is a quasi-closed set of $C$ if and only if it is of the form $Z \cup \{x, y\}$ where $Z$ is a union of a proper subset of the connected components of $G_P(x, y)$.

We prove the if part.
Assume that $Q = \{x, y\} \cup Z$ where $Z$ is a union of a proper subset of the connected components of $G_P(x, y)$.
Following Lemma~\ref{lem:wild-saturation-poly}, to show that $Q$ is quasi-closed, we need only show that $Q$ satisfies all the implications $uv \imp [u, v] \setminus \{u, v\}$ of $\ib_c$ where $\cl(uv) \subset \cl(xy)$.
We have two cases: $\{u, v\} \cap \{x, y\} = \emptyset$ or $\{u, v\} \cap \{x, v\} \neq \emptyset$ in which case the intersection contains exactly one element.
We consider the first case.
The assumption $\cl(uv) = [u, v] \subset [x, y] = \cl(xy)$ entails $u, v \in [x, y]$.
Moreover, there exists a connected component $D$ of $G_P(x, y)$ that includes $[u, v]$ as any element of $[u, v]$ is comparable to $u$ and $v$.
Since $u, v \in Q$ and $Z$ is a union of connected components of $G_P(x, y)$, we deduce $D \subseteq Q$ and hence $[u, v] \subseteq Q$.
Hence, $Q$ satisfies the implication $uv \imp [u, v] \setminus \{u, v\}$.
We turn to the second case.
Let us assume $u = x$ and $y \neq v$.
Then, any element in $[x, v]$ is in the connected component induced by $v$ and a similar argument shows that $Q$ satisfies the implication $xv \imp [x, v] \setminus \{x, v\}$.
We deduce by Lemma~\ref{lem:wild-saturation-poly} that $Q$ is indeed quasi-closed.

We turn to the only if part, which we show using contrapositive.
Let $\{x, y\} \subseteq Y \subseteq C$ be such that there exists a connected component $D$ of $G_P(x, y)$ satisfying $Y \cap D \neq \emptyset$ but $D \nsubseteq Y$.
In particular, in $D$ there exist adjacent comparable vertices $u$ and $v$ such that $u \in Y \cap D$ and $v \notin Y$.
Then, we have either $v \in [x, u]$ or $v \in [u, y]$.
Without loss of generality assume $v \in [x, u]$.
Since $[x, u] \subset [x, y]$ by definition, we deduce that $\{u, x\} \subseteq Y$, $\cl(xu) \subset \cl(xy) = \cl(Y)$ but $\cl(xu) \nsubseteq Y$ as $v \notin Y$.
Therefore, $Y$ is not quasi-closed, which concludes the proof.
\end{proof}

\begin{example} \label{ex:poset-cg-optim}
We continue Example~\ref{ex:poset-cg}.
According to Lemma~\ref{lem:posets-qc} the edges of $\HQC(\U)$ are exactly the connected components of $G_P(x, y)$: $\lset{abgf}$, $d$, and $\lset{ceh}$ (see Figure~\ref{fig:poset-cg}).
They are illustrated in Figure~\ref{fig:poset-cg-optim}, as well as the corresponding inclusion-wise maximal quasi-closed sets that are thus obtained by taking all connected components of $G_P(x, y)$ but one.
The set $\lset{ade}$, also highlighted in the figure, is a minimal (hence minimum) hitting set of $\HQC(\U)$.
This makes the implication $\lset{xy} \imp \lset{ade}$ optimum for this equivalence class.
An example of optimum base $(\U, \ib)$ for $\mathrm{Co}(P)$ is:
\[ 
\ib = \left\{\begin{array}{lllll}
\mathit{xf} \imp \mathit{ab}, & \mathit{xg} \imp \mathit{ab}, & \mathit{xe} \imp c, & \mathit{xh} \imp \mathit{c}, & \mathit{xy} \imp \mathit{ade}, \\
\mathit{ay} \imp \mathit{fg}, & \mathit{by} \imp \mathit{fg}, & \mathit{cy} \imp \mathit{e}, & \mathit{ey} \imp \mathit{h}, & \mathit{ch} \imp e
\end{array}
\right\}.
\]
\begin{figure}[ht!]
\centering 
\includegraphics[page=2, width=\linewidth]{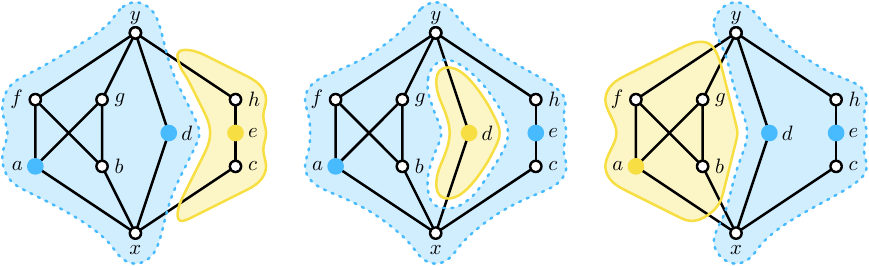}
\caption{
The 3 edges of of the quasi-closed hypergraph associated to $\U$ in the double-shelling convex geometry of Example~\ref{ex:poset-cg} (solid yellow zones).
The dotted blue zones thus indicate the corresponding quasi-closed sets.
The set $\mathit{ade}$ (bold nodes) form a minimum hitting set of this hypergraph.
}
\label{fig:poset-cg-optim}
\end{figure}
\end{example}

Hence the edges of $\HQC(C)$ are indeed disjoint for any essential set $C$.
More precisely, an optimum implication for $C$ is obtained by picking one element in each connected component of the corresponding graph, as proved by Adaricheva~\cite[Lemma 26]{adaricheva2017optimum}.
This proves item (1) of Theorem~\ref{thm:cg-classes}.

\subsection{\texorpdfstring{$D$-geometries}{D-geometries}} \label{subsec:D-geometries}

We proceed to analyze $D$-geometries.
We first introduce $D$-generators and the $D$-relation following \cite{adaricheva2025computing}.
Let $(\U, \cs)$ be a standard closure system, i.e., a closure system satisfying $\emptyset \in \cs$ and $\cl(x) \setminus \{x\} \in \cs$ for every $x \in \U$.
Observe that convex geometries are standard.
The \emph{binary part} of $\cl$ is the operator $\cl^b$ given by $\cl^b(Y) = \bigcup_{y \in \U} \cl(y)$ for every $Y \subseteq \U$.
Observe that $\cl^b$ is a closure operator and the family $\{x \imp y \st y \in \cl(x)\}$ constitutes an implicational base of the corresponding closure system.
Let $x \in \U$.
A \emph{$D$-generator} of $x$ is a minimal generator $A$ of $x$ such that $x \notin \cl^b(A)$  and for every other minimal generator $A'$ of $x$, $A' \subseteq \cl^b(A)$ implies $A' = A$.
Note that the minimal generator $A$ is a $D$-generator of $x$ if and only if $x \notin \cl^b(A)$, $x \in \cl(A)$ and for every $A' \subset \cl^b(A)$, $x \notin \cl(A')$.
$D$-generators allow to define the \emph{$D$-relation} of $(\U, \cs)$.
It is a binary relation over $\U$ where $x D a$ holds if there exists a $D$-generator of $x$ that contains $a$. 
The $D$-relation is acyclic if its corresponding directed graph is.

Let $(\U, \cs)$ be a convex geometry.
We start with a simple observation.

\begin{proposition} \label{prop:unit-mingen}
Let $x, y \in \U$, $x \neq y$.
If $x \in \cl(y)$ then $x$ does not belong to any minimal generator of $y$.
\end{proposition}

\begin{proof}
Assume for contradiction there exists a minimal generator $A$ of $y$ that contains $x$.
Then, $x, y \notin \cl(A \setminus \{x\})$, $x \in \cl((A \setminus \{x\}) \cup \{y\})$, $y \in \cl(A)$.
This contradicts $\cl$ being antiexchange.
\end{proof}

Then, we relate maximal quasi-closed sets with the $D$-relation.

\begin{proposition} \label{prop:quasi-D}
Let $C$ be essential with at least two extreme points and let $Q$ be an inclusion-wise maximal quasi-closed set of $C$.
Let $x_1, \dots, x_k$ be the $\cl^b$-minimal elements of $C \setminus Q$.
If $k \geq 2$, then for any $x_i$ there exists $x_j$ different from $x_i$ such that $x_j D x_i$.
\end{proposition}

\begin{proof}
Assume that $k \geq 2$ and consider $Q \cup \{x_i\}$ for some $1 \leq i \leq k$.
As $Q$ is maximal quasi-closed and $Q \setminus C$ contains at least $2$ elements, $Q \cup \{x_i\}$ is neither closed nor quasi-closed.
Hence, there exists $A \subseteq Q \cup \{x_i\}$ such that $\cl(A) \subset C$ and $\cl(A) \nsubseteq Q \cup \{x_i\}$.
Since $Q$ is quasi-closed, $x_i \in A$ must hold.
Let $y \in Q \setminus C$ such that $y \in \cl(A)$.
Without loss of generality, $A$ can be chosen to be a minimal generator of $y$ containing $x_i$.
By Proposition~\ref{prop:unit-mingen}, $x_i \notin \cl(y)$.
Thus, there exists $x_j$ different from $x_i$ such that $x_j \in \cl(y) \subseteq \cl(A) \subset C$.
It remains to argue that $A$ can be further refined to a $D$-generator of $x_j$ containing $x_i$.
Let $A'$ be any subset of $\cl^b(A)$ that does not contain $x_i$.
Since $C \neq \cl(x)$ for every $x \in C$ by assumption, $\cl(x) = \cl^b(x) \subset C$ must hold.
We deduce that $A' \subseteq Q$ must hold.
Since $A' \subseteq \cl^b(A \setminus \{x_i\}) \subseteq \cl(A) \subset C$, we deduce that $\cl(A') \subseteq Q$ must hold and hence $x_j \notin \cl(A')$.
We conclude that any $D$-generator of $x_j$ contained in $\cl^b(A)$ must contain $x_i$.
We deduce $x_j D x_i$ as required.
\end{proof}

We are now ready to investigate $D$-geometries.
The convex geometry $(\U, \cs)$ is a \emph{$D$-geometry} if its $D$-relation is acyclic.
This class properly contains so called \emph{acyclic convex geometries}, that have been well-studied, see, e.g.,~\cite{hammer1995quasi,wild1994theory,defrain2021translating,santocanale2014lattices}.
A convex geometry is acyclic if its $\delta$-relation is, where $\delta$ is a binary relation over $\U$ where $x \delta a$ if there exists a minimal generator of $x$ that contains $A$.
Observe that $D \subseteq \delta$, which straightforwardly shows that acyclic convex geometries are indeed $D$-geometries.
We give below an example of $D$-geometry, smaller than the one proposed in \cite{adaricheva2017optimum}, that is neither acyclic nor a double-shelling convex geometry.
\begin{example} \label{ex:D-geometry}
Let $\U = \{a, b, c, d, e\}$ and let $\ib = \{ \lset{ab} \imp c, \lset{cd} \imp e, e \imp a, d \imp a\}$.
The corresponding closure system $(\U, \cs)$ is a convex geometry and is depicted on the left of Figure~\ref{fig:D-geometry}.
Among the elements of $\U$, only $c$ and $e$ have minimal generators of size at least $2$: $c$ has $\lset{ab}$, $\lset{bd}$ and $\lset{be}$ while $e$ has $\lset{cd}$ and $\lset{bd}$.
Since $a \in \cl, ^b(d)$ and $a \in \cl^b(e)$, we have $\cl^b(ab) \subset \cl^b(bd), \cl^b(be)$ which makes $ab$ the unique $D$-generator of $a$.
Both minimal generators of $e$ are $D$-generators as their closure with respect to $\cl^b$ are incomparable.
We thus have $c D a, b$ and $e D b, c, d$.
This makes the $D$-relation acyclic and $(\U, \cs)$ a $D$-geometry.
For convenience, the directed graph associated to the $D$-relation is pictured on the right of Figure~\ref{fig:D-geometry}.
Observe that $(\U, \cs)$ cannot be a double-shelling convex geometry as some singletons are not closed.
It is not acyclic as $cd$ is a minimal generator of $e$ and $be$ a minimal generator of $c$.
\begin{figure}[ht!]
    \centering
    \includegraphics{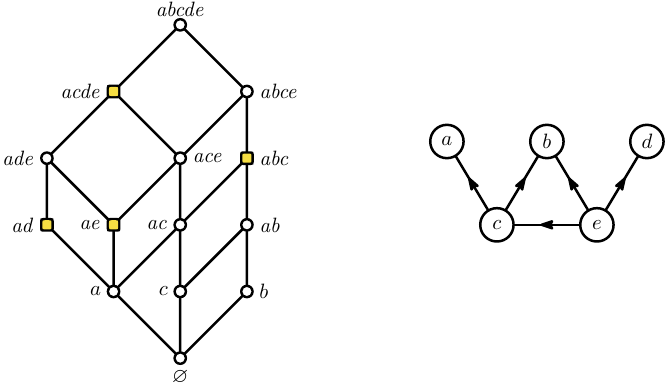}
    \caption{On the left, the $D$-geometry of Example~\ref{ex:D-geometry}. Square (yellow) nodes are essential closed sets. On the right, the directed graph of its associated $D$-relation. We have for instance $e D b$ as $bd$ is a $D$-generator of $e$. The graph, hence the $D$-relation, is indeed acyclic.}
    \label{fig:D-geometry}
\end{figure}

\end{example}

Let us assume then that $(\U, \cs)$ is a $D$-geometry and let $C$ be an essential set that is not the closure of any element $x \in \U$.
Since $(\U, \cs)$ does not have $D$-cycles by assumption, $C \setminus Q$ has exactly one $\cl^b$-minimal element $x$.
By combining Propositions~\ref{prop:join-qc-disjoint}, \ref{prop:quasi-D} and \ref{prop:unit-mingen}, we thus obtain:

\begin{corollary} \label{cor:D-geom-qc}
Let $(\U, \cs)$ be a $D$-geometry.
For any essential set $C$, $\HQC(C)$ has singleton edges.
Therefore, $(\U, \cs)$ has a unique optimum implicational base.
\end{corollary}

Corollary~\ref{cor:D-geom-qc} proves item (2) of Theorem~\ref{thm:cg-classes}.
Note moreover that the unique optimum base of $(\U, \cs)$ is the so-called optimized $E$-base, which was shown to be optimum in~\cite{adaricheva2017optimum}.

\subsection{Convex geometries with the Carousel property} \label{subsec:carousel}

Let $(\U, \cs)$ be a convex geometry.
Following~\cite{adaricheva2017optimum}, we say that $(\U, \cs)$ has the \emph{$n$-Carousel property}, $n \in \mathbb{N}$, if for every $A \subseteq \U$ with at least two elements, and for every $x, y \in \cl(A)$, there exists $A' \subset A$ such that $x \in \cl(A' \cup \{y\})$ and $\card{A'} \leq \min \{n, \card{A} - 1\}$.

\begin{proposition} \label{prop:carousel-qc-disjoint}
Let $C$ be an essential set such that $\ex(C)$ has at least two elements.
Then, $\ex(C)$ is the unique quasi-closed set of $C$.
Therefore, $C \setminus \ex(C)$ is the unique edge of $\HQC(C)$.
\end{proposition}

\begin{proof}
Since $C$ is assumed to be essential, it admits at least one quasi-closed set.
Let $P := \ex(C)$.
As the saturation operator $\sigma$ is a closure operator it is sufficient, in order to prove the claim, to show that for any $x \in C \setminus P$, $\sigma(P \cup\{x\}) = C$.
the result readily holds if $C \setminus P = \{x\}$ for some $x$.
Consider now that $C \setminus P = \{x_1, \dots, x_k\}$ for some $k \geq 2$.
We show that $x_j \in \sigma(P \cup \{x_i\})$ for every distinct $x_i, x_j$.
We have $\card{P} \geq 2$ and $x_i, x_j \in \cl(P)$ by assumption.
Therefore, applying the $n$-Carousel property, we deduce that there exists $P' \subset P$, with size at most $\min \{n, \card{P} - 1\}$ such that $x_j \in \cl(P' \cup \{x_i\})$.
Given that $P' \cup \{x_i\} \subset P \cup \{x_i\}$, we are only left to show that $\cl(P' \cup \{x_i\}) \subset C$ in order to deduce that $x_j \in \sigma(P \cup \{x_i\})$.
As $P' \subset P$, there exists $z \in P \setminus P'$ such that $P' \cup \{x_i\} \subseteq C \setminus \{z\}$ with $C \setminus \{z\} \in \cs$.
We deduce $\cl(P' \cup \{x_i\}) \subseteq C \setminus \{z\} \subset C$ as expected.
Hence, for any $x_j \neq x_i$, $x_j \in \sigma(P \cup \{x_i\})$.
We obtain that for any $x_i$, $\sigma(P \cup \{x_i\}) = C$.
It follows that $P$ is the unique quasi-closed set of the essential set $C$.
\end{proof}

Combining Propositions~\ref{prop:join-qc-disjoint} and \ref{prop:carousel-qc-disjoint}, we obtain that convex geometries with the Carousel property do have quasi-closed hypergraphs with disjoint edges, which proves item (3) of Theorem~\ref{thm:cg-classes}.
Affine convex geometries, i.e., convex geometries arising from point configurations in $\mathbb{R}^d$~\cite{edelman1985theory}, have been shown to enjoy the Carousel property~\cite{adaricheva2017optimum}.
Since in an affine convex geometry each singleton is closed we further have:

\begin{corollary}
If $(\U, \cs)$ is an affine convex geometry, then $\HQC(C)$ has a unique edge for every essential set $C$.
\end{corollary}

\begin{remark}
$q$-acceptant are not $q$-Carousel as shown by the closure system represented by the implications $ad \imp b$, $ac \imp b$, $cd \imp ab$.
\end{remark}

\subsection{Acceptant convex geometries} \label{subsec:q-acc-cg}

We finally study ($q$-)acceptant convex geometries.
A convex geometry $(\U, \cs)$ is \emph{$q$-acceptant} for some $1 \leq q \leq \card{\U}$ if for every $C \in \cs$, $\card{\ex(C)} = \min(q, \card{C})$.
Observe that a convex geometry is acceptant for at most one $q$.
Besides, the definition shows that a convex geometry $(\U, \cs)$ is $q$-acceptant if and only if every subset of $\U$ with strictly less than $q$ elements is closed and every closed set with at least $q$ elements has precisely $q$ extreme points.
Moreover, we readily have that all minimal generators of a $q$-acceptant have size $q$.
The trivial closed sets of a $q$-acceptant convex geometry thus are precisely all the closed sets of size at most $q$.
Two $2$-acceptant convex geometries are picture in Figure~\ref{fig:example-2-acceptant}.

\begin{figure}[ht!]
\centering 
\includegraphics[width=0.8\linewidth]{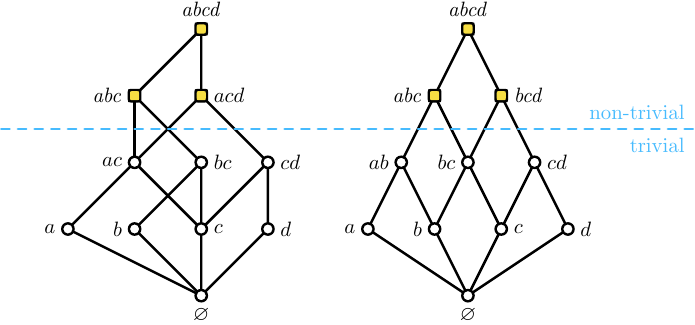}
\caption{Examples of $2$-acceptant convex geometries.
The dotted lines separate non-trivial and trivial closed sets.
Essential closed sets are indicated by (yellow) square nodes.
Closed sets with at most $2$ elements are trivial, other ones are non-trivial and all have exactly $2$ extreme points.
Notice that non-trivial and essential closed sets coincide.}
\label{fig:example-2-acceptant}
\end{figure}

We prove that, much as affine convex geometries, an essential set $C$ of a $q$-acceptant convex geometry has a unique maximal quasi-closed set.

\begin{lemma} \label{lem:q-acc-unique-qc}
Any essential set $C$ has a unique maximal quasi-closed set.
\end{lemma}

\begin{proof}
Let $C$ be essential and assume for contradiction it admits at least two distinct inclusion-wise maximal quasi-closed set $Q_1$, $Q_2$.
Observe that $C$ is non-trivial.
Let $C'$ be a predecessor of $C$.
We have $C' = C \setminus \{x\}$ for some $x \in C$.
Such a $C'$ must exist as $C$ is assumed to be non-trivial, hence non-empty.
As $Q_1, Q_2$ are quasi-closed, $C' \cap Q_1$ and $C' \cap Q_2$ are closed.
Observe that since $x \in Q_1, Q_2 \subseteq C$, $C' = C \setminus \{x\}$ and $Q_1$, $Q_2$ and $C'$ are pairwise incomparable, $C' \cap Q_1$ and $C' \cap Q_2$ are also incomparable.
As $\ex(C) \setminus \{x\} \subseteq C', Q_1, Q_2$, we have that $\ex(C) \setminus \{x\} \subseteq \ex(C'),  \ex(C' \cap Q_1)$, $\ex(C' \cap Q_2)$, that is, $C'$, $C' \cap Q_1$ and $C' \cap Q_2$ share $q-1$ extreme points of $C$ as $(\U, \cs)$ is $q$-acceptant.
Now let $C'' := (C' \cap Q_1) \lor (C' \cap Q_2)$.
Since $C' \cap Q_1 \subseteq C'' \subseteq C'$, we have by previous discussion that $C''$ have at least $q-1$ extreme points given by $\ex(C') \setminus \{x\}$.
Since $C''$ is the supremum of $(C' \cap Q_1) \lor (C' \cap Q_2)$ there must exist a predecessor $C_1$ of $C''$ such that $(C' \cap Q_1) \subseteq C_1$ and $(C' \cap Q_2) \nsubseteq C_1$.
Because $(\U, \cs)$ is a convex geometry, we have $C_1 = C'' \setminus \{x_2\}$ for some $x_2 \in (C' \cap C_2) \setminus (C' \cap C_1)$.
Dually, we derive the existence of some $x_1 \in (C' \cap Q_1) \setminus (C' \cap Q_2)$ such that $C_2 = C'' \setminus \{x_1\}$ is a predecessor of $C''$ that includes $C' \cap Q_2$ but not $C' \cap Q_1$.
As $C_1, C_2$ are predecessors of $C''$, $x_1, x_2 \in \ex(C'')$ holds.
Moreover, $x_1, x_2 \notin \ex(C') \setminus \{x\}$ holds by above discussion.
We deduce that $C''$ has at least $q+1$ extreme points, a contradiction with $(\U, \cs)$ being $q$-acceptant.
This concludes the proof.
\end{proof}

By Lemma~\ref{lem:q-acc-unique-qc}, we have that $\mathcal{Q}(C)$ has, for every essential $C$, a unique edge.
Its minimum hitting sets are therefore singletons, much as affine convex geometries, which finally proves item (4) of Theorem~\ref{thm:cg-classes}.

\begin{example} \label{ex:acceptant}
Consider the two $2$-acceptant convex geometries of Figure~\ref{fig:example-2-acceptant}.
Let us call $(\U, \cs_1)$ the left one and $(\U, \cs_2)$ the right-one.
In both cases, non-trivial closed sets and essential sets coincide.
For $\cs_1$ we have $\HQC(abc) = \HQC(acd) = \{c\}$ and $\HQC(abcd) = \{a\}$.
For $\cs_2$ we have $\HQC(abc) = \{b\}$, $\HQC(bcd) = \{c\}$ and $\HQC(abcd) = \{bc\}$.
We thus observe that $(\U, \cs_1)$ has a unique optimum base while $(\U, \cs_2)$ has two depending on which of $b$ and $c$ we choose as the conclusion associated to $ad$.
\end{example}


\section{Discussion} \label{sec:discussion}

In this paper, we proposed a characterization of the optimum bases of a convex geometry.
This characterization relies on quasi-closed hypergraphs, the hypergraphs defined for each essential set by taking the complement of its inclusion-wise maximal quasi-closed sets.
In the case where each quasi-closed hypergaph of the convex geometry has disjoint edges, we proved that using existing minimization and reduction algorithms was sufficient to produce, in polynomial time, an optimum base out of any other implicational base.
This property of quasi-closed hypergraphs turns out to apply to double-shelling, acyclic, affine and acceptant convex geometries, which allows to unify existing results of the literature.

We conclude the paper with some observations and discussions for future work.
First, we observe that while in a convex geometry the property of having quasi-closed hypergraphs with disjoint edges guarantees efficient optimization, whether this property can be recognized in polynomial time from an implicational base remains unsettled.
This leads to the following question:
given an implicational base $(\U, \ib)$ of some convex geometry $(\U, \cs)$, is it possible to decide in polynomial time whether all quasi-closed hypergraphs of $(\U, \cs)$ have disjoint edges?

Another, more general, intriguing questions consists in using our approach to identify further classes of convex geometries for which optimization is tractable.
To conclude the paper, we observe that using our approach to identify further classes of convex geometries for which optimization is tractable seems to be an intriguing problem for further research.
Indeed, by Theorem~\ref{thm:cg-optim-ib}, a convex geometry can be optimized efficiently if and only if for each quasi-closed hypergraph, some minimum hitting set can be computed in polynomial time.
Thus, classifying convex geometries according to their quasi-closed hypergraphs may be a good yardstick to better understand the hardness of the optimization task in convex geometries.
In particular, we may study whether other classes of convex geometries have quasi-closed hypergraphs with disjoint edges.

\bibliographystyle{alpha}
\bibliography{biblio}

\appendix 

\section{Other inclusion relations between classes and properties} \label{subsec:hierarchy}

Here we complete the proof that the hierarchy presented in Figure~\ref{fig:hierarchy} is correct.
Adopting the same names as in the figure, we already have that:
\begin{itemize} 
\item unique edge $\cup$ singleton edges $\subseteq$ unique or singleton edges $\subseteq$ disjoint edges $\subseteq$ convex geometries.
The last inclusion follows by definition of disjoint edges in Figure~\ref{fig:hierarchy} while others are clear; 
\item (affine $\subset$ Carousel property) and (acyclic $\subset$ $D$-geometries) follows from~\cite{adaricheva2017optimum};
\item all other inclusions (non-strict) have been proved in Theorem~\ref{thm:cg-classes}.
\end{itemize} 
It remains to argue about strict inclusions and that no other inclusion holds.

We start with a first series of examples that establishes the strict inclusions that are not already known:
\begin{enumerate}[label=(\arabic*)]
    \item Example~\ref{ex:gc-non-disjoint} shows disjoint edges $\subset$ convex geometries; 
    \item Example~\ref{ex:disjoint-non-unique-singleton-poset} shows (unique or singleton edges $\cup$ double-shelling) $\subset$ disjoint edges;
    \item Example~\ref{ex:singleton-unique-non-carousel-unique-singleton} shows (Carousel property $\cup$ unique edge $\cup$ singleton edges) $\subset$ unique or singleton edges; 
    \item Example~\ref{ex:unique-non-affine-acceptant} shows affine $\cup$ acceptant $\subset$ unique edges;
    \item Example~\ref{ex:singleton-non-D} shows $D$-geometries $\subset$ singleton edges.
\end{enumerate}

\begin{example} \label{ex:gc-non-disjoint}
Let $\U = \{a, b, c, d, e\}$ and let $\ib = \{\lset{ae} \imp \lset{cd}$, $\lset{bcd} \imp \lset{e}$ and $\lset{ab} \imp \lset{cde}\}$.
The corresponding closure system $(\U, \cs)$ is a convex geometry depicted in Figure~\ref{fig:gc-non-disjoint}.
The quasi-closed hypergraphs are:
\[
\HQC(\U) = \{de, ce\} \quad\quad \HQC(acde) = \{c, d\} \quad\quad \HQC(bcde) = \{e\}.
\]
In particular, the edges of $\HQC(\U)$ are not disjoint.
\begin{figure}[ht!]
    \centering
    \includegraphics[page=1]{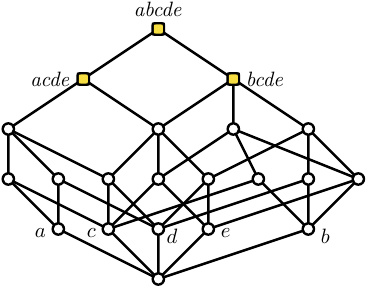}
    \caption{The convex geometry of Example~\ref{ex:gc-non-disjoint}. 
    Some labels are omitted for clarity.
    Essential closed sets are indicated by (yellow) square nodes.
    The convex geometry has a quasi-closed hypergraph whose edges are not disjoint, namely $\HQC(\U) = \{de, ce\}$.}
    \label{fig:gc-non-disjoint}
\end{figure}
\end{example}

\begin{example} \label{ex:disjoint-non-unique-singleton-poset}
Let $\U = \{a, b, c, d, e\}$ and let $\ib = \{\lset{acd} \imp e, \lset{bce} \imp d, \lset{ab} \imp \lset{cde}\}$.
The associated closure system $(\U, \cs)$ is a convex geometry depicted in Figure~\ref{fig:disjoint-non-unique-singleton-poset}.
The quasi-closed hypergraphs are:
\[\HQC(\U) = \{c, \lset{de}\} \quad\quad \HQC(\lset{acde}) = \{d\} \quad\quad \HQC(\lset{bcde}) = \{e\}. \]
All of them have disjoint edges.
In particular, $\HQC(\U)$ has two edges, one of them not being a singleton.
Moreover, $(\U, \cs)$ is not a double-shelling convex geometry has it admits minimal generators of size $3$, e.g., $\lset{acd}$ being a minimal generator of $e$, which cannot occur in double-shelling convex geometries~\cite{kashiwabara2010characterizations}.
\begin{figure}[ht!]
    \centering
    \includegraphics[page=2]{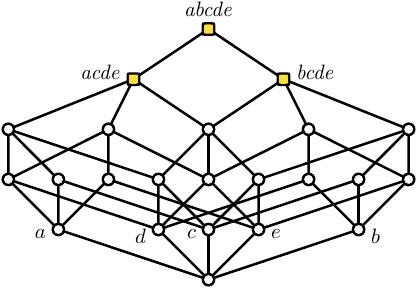}
    \caption{The convex geometry of Example~\ref{ex:disjoint-non-unique-singleton-poset}. 
    Some labels are omitted for clarity.
    Essential closed sets are indicated by (yellow) square nodes.
    All quasi-closed hypergraphs have disjoint edges but the convex geometry is not a double-shelling convex geometry, $\HQC(\U) = \{c, de\}$ has two disjoint edges one of them not being a singleton.}
    \label{fig:disjoint-non-unique-singleton-poset}
\end{figure}
\end{example}

\begin{example} \label{ex:singleton-unique-non-carousel-unique-singleton}
Let $\U = \{a, b, c, d, e, f\}$ and let $\ib = \{\lset{ab} \imp \lset{cdef}, \lset{cf} \imp \lset{de}, \lset{cd} \imp e, \lset{ef} \imp d\}$.
The associated convex geometry $(\U, \cs)$ is given in Figure~\ref{fig:singleton-unique-non-carousel-unique-singleton}.
We have:
\[ \HQC(\U) = \{c, f\} \quad\quad \HQC(\lset{cdef}) = \{\lset{de}\} \quad\quad \HQC(\lset{cde}) = \{e\} \quad\quad \HQC(\lset{def}) = \{d\} \]
There is a quasi-closed hypergraph that has a unique edge which is not a singleton, and a quasi-closed hypergraph with two edges being singletons.
Moreover, $(\U,  \cs)$ does not satisfy the Carousel property for any $n$ as we have $c, d \imp \cl(ab)$ but $c \notin \cl(A' \cup \{d\})$ for any $A' \subset \lset{ab}$.
\begin{figure}[ht!]
    \centering
    \includegraphics[page=3]{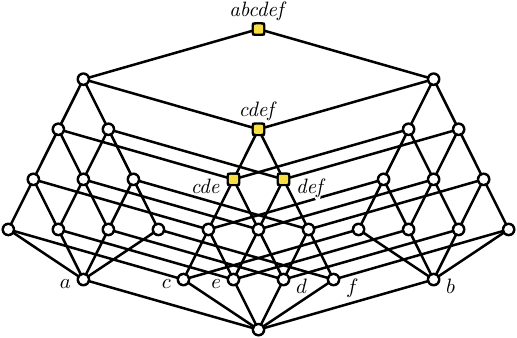}
    \caption{The convex geometry of Example~\ref{ex:singleton-unique-non-carousel-unique-singleton}. 
    Some labels are omitted for clarity.
    Essential closed sets are indicated by (yellow) square nodes.
    This convex geometry does not have the Carousel property.
    We also have $\HQC(\U) = \{c, f\}$ which has two singleton edges and $\HQC(\lset{cdef}) = \{\lset{de}\}$ which has one non-singleton edge.
    }
    \label{fig:singleton-unique-non-carousel-unique-singleton}
\end{figure}
\end{example} 

\begin{example}  \label{ex:unique-non-affine-acceptant}
Let $\U = \{a, b, c, d\}$ and $\ib =\{ ab \imp c, b \imp d \}$.
The closure system $(\U, \cs)$ is acyclic with two quasi-closed hypergraphs:
\[\HQC(\U) = \{c\} \quad\quad \HQC(bd) = \{d\} \] 
Both have a unique (singleton) edge.
The convex geometry $(\U, \cs)$ is not affine as the singleton $\{b\}$ is not closed.
As for $q$-acceptance, $b \imp d$ implies that $q = 1$ and $ab \imp c$ that $q = 2$.
Since a convex geometry can be acceptant for a unique $q$, $(\U, \cs)$ is not acceptant either.
\begin{figure}[ht!]
    \centering
    \includegraphics[page=4]{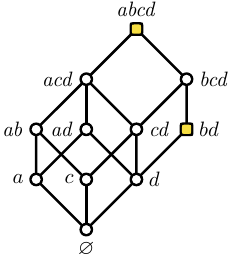}
    \caption{The convex geometry of Example~\ref{ex:unique-non-affine-acceptant}. 
    Essential closed sets are indicated by (yellow) square nodes.
    The two quasi-closed hypergraphs have a unique (singleton) edge but the convex geometry is neither affine nor acceptant.}
    \label{fig:unique-non-affine-acceptant}
\end{figure}
\end{example}

\begin{example} \label{ex:singleton-non-D}
Let $\U = \{a, b, c, d, e\}$ and $\ib = \{ad \imp e, e \imp bc, ac \imp b, bd \imp c\}$.
The corresponding convex geometry $(\U, \cs)$ is depicted in Figure~\ref{fig:singleton-non-D}.
The quasi-closed hypergraphs are:
\[ \HQC(\U) = \{e\} \quad\quad \HQC(\lset{bce}) = \{b, c\} \quad\quad \HQC(\lset{abc}) = \{b\} \quad\quad \HQC({bcd}) = \{c\} \] 
All of the quasi-closed hypergraphs have singleton edges.
Yet, there is a $D$-cycle due to $ac \imp b$ and $bd \imp c$.
Thus $(\U, \cs)$ is not a $D$-geometry.
\begin{figure}[ht!]
    \centering
    \includegraphics[page=5]{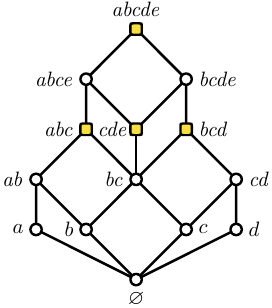}
    \caption{The convex geometry of Example~\ref{ex:singleton-non-D}. 
    Essential closed sets are indicated by (yellow) square nodes.
    All quasi-closed hypergraphs have singleton edges and the convex geometry has a $D$-cycle.
    It is thus not a $D$-geometry.}
    \label{fig:singleton-non-D}
\end{figure}
\end{example}

We now proceed to a second series of example proving that no other inclusion holds in Figure~\ref{fig:hierarchy}:
\begin{enumerate}[label=(\arabic*)]
    \item Example~\ref{ex:poset-non-unique-singleton} shows double-shelling $\nsubseteq$ unique or singleton edges;
    \item Example~\ref{ex:affine-non-acceptant-singleton-double} shows affine $\nsubseteq$ acceptant $\cup$ singleton edges $\cup$ double-shelling;
    \item Example~\ref{ex:acceptant-non-affine-singleton-double} shows acceptant $\nsubseteq$ affine $\cup$ singleton edges $\cup$ double-shelling;
    \item Example~\ref{ex:acyclic-non-carousel-unique-double} shows acyclic $\nsubseteq$ Carousel property $\cup$ unique edge $\cup$ double-shelling.    
\end{enumerate}
All other incomparabilities follows from the above four.

\begin{example} \label{ex:poset-non-unique-singleton}
Let $\U = \{a, b, c, d, e\}$ and let $P = (\U, \leq)$ be the poset given by $a \leq b \leq c \leq d$ and $a \leq e \leq d$.
It is pictured on the left of Figure~\ref{fig:poset-non-unique-singleton}.
In the corresponding double-shelling convex geometry $\mathrm{Co}(P)$ we have:
\[\HQC(\U) = \{\lset{bc}, e\} \quad\quad \HQC(\lset{ac}) = \{b\} \quad\quad \HQC(\lset{bd}) = \{d\} \] 
In particular $\HQC(\U)$ has two edges, one of which is not a singleton.
\begin{figure}[ht!]
\centering 
\includegraphics[page=1]{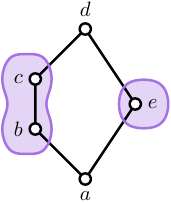}
\caption{On the left, the poset of Example~\ref{ex:poset-non-unique-singleton}.
The two edges $bc$ and $e$ of $\HQC(\U)$ in the corresponding double-shelling convex geometry are circled in purple.}
\label{fig:poset-non-unique-singleton}
\end{figure}
\end{example}

\begin{example} \label{ex:affine-non-acceptant-singleton-double}
Let $\U = \{a, b, c, d, e, f\}$ and consider any realization in $\mathbb{R}^2$ that consisting in putting $a, b, c, d$ on a straight line and constructing a triangle $\lset{cde}$ containing $f$.
The construction is pictured in Figure~\ref{fig:affine-non-acceptant-singleton-double}.
In this configuration, $\lset{cde}$ is a minimal generator of $f$ of size $3$ and $ad$ a minimal generator of $b$ of size $2$.
We deduce that $(\U, \cs)$ is not acceptant as $q = 2$ and $q = 3$ should hold.
Moreover, minimal generators of double-shelling convex geometries have size at most $2$.
Hence, $(\U, \cs)$ is not double-shelling.
As for the quasi-closed hypergraphs we have:
\[\HQC(\lset{abc}) = \{b\} \quad\quad \HQC(\lset{bcd}) = \{c\} \quad\quad \HQC(\lset{abcd}) = \{bc\} \quad\quad \HQC(\lset{cdef}) = \{f\}\] 
Hence, there exists a quasi-closed hypergraph that does not have singleton edges, namely, $\HQC(\lset{abcd})$.
\begin{figure}[ht!]
\centering 
\includegraphics[page=2]{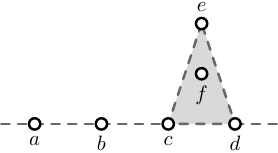}
\caption{The point configuration in $\mathbb{R}^2$ of Example~\ref{ex:affine-non-acceptant-singleton-double}.
The convex hull of the triangle $\lset{cde}$ contains $f$.}
\label{fig:affine-non-acceptant-singleton-double}
\end{figure}
\end{example}

\begin{example} \label{ex:acceptant-non-affine-singleton-double}
Consider the $2$-acceptant convex geometry on $\U = \{a, b, c, d, e\}$ depicted in Figure~\ref{fig:acceptant-non-affine-singleton-double}.
Its canonical base $(\U, \ib_c)$ is given by $\ib_c = \{ad \imp b, be \imp d, ce \imp d, ae \imp bd, bc \imp de, ac \imp bde\}$.
We have:
\[
\begin{array}{l l l}
\HQC(\lset{abd}) = \{b\}   & \HQC(\lset{bde}) = \{d\}   & \HQC(\lset{cde}) = \{d\} \\ 
\HQC(\lset{abde}) = \{bd\} & \HQC(\lset{bcde}) = \{de\} & \HQC(\U) = \{bde\}
\end{array}
\]
The quasi-closed hypergraphs of $\lset{abde}$, $\lset{bcde}$ and $\U$ witness that not all quasi-closed hypergraphs have singleton edges.
We argue that $(\U, \cs)$ is neither an affine nor a double-shelling convex geometry.
In order to be affine, $a, b, d, e, c$ must be aligned on a line in this order.
This follows from $bde \in \cl(ac)$, $d \in \cl(be)$, $b \in \cl(ad)$ and $de \in \cl(bc)$.
Hence, $e \in \cl(cd)$ should hold.
However we have $\cl(cd) = cd$, which contradicts that $(\U, \cs)$ is affine.
The same reasoning holds to show that $(\U, \cs)$ is not a double-shelling convex geometry: $a, b, d, e, c$ must form a chain, in this order or its dual.
Since $e \notin \cl(cd)$, $(\U, \cs)$ cannot be double-shelling either.
\begin{figure}[ht!]
\centering 
\includegraphics[page=3]{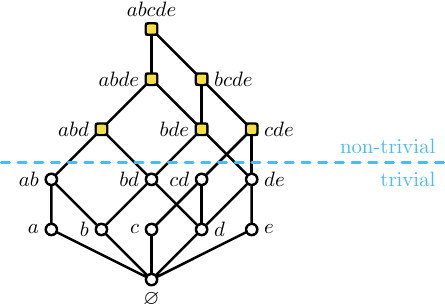}
\caption{The $2$-acceptant convex geometry of Example~\ref{ex:affine-non-acceptant-singleton-double}.
Essential closed sets are represented with (yellow) square nodes.
The dotted line indicates the separation between trivial and non-trivial closed sets.} 
\label{fig:acceptant-non-affine-singleton-double}
\end{figure}
\end{example}

\begin{example} \label{ex:acyclic-non-carousel-unique-double}
We consider a slight variation of Example~\ref{ex:unique-non-affine-acceptant}.
Namely, let $\U = \{a, b, c, d, e\}$ and let $\ib = \{ab \imp ce, b \imp d\}$.
Since $(\U, \ib)$ has no cycle, $(\U, \cs)$ is acyclic.
We have $\HQC(\U) = \{c, e\}$ and $\HQC(\lset{bd}) = \{d\}$.
The hypergraph $\HQC(\U)$ has two edges. 
As $b$ is a minimal generator of $d$ of size $1$, $(\U, \cs)$ cannot be a double-shelling convex geometry.
The Carousel cannot hold either: we have $cd \in \cl(ab)$ but $c \notin \cl(ad), \cl(bd)$.

\end{example}

\end{document}